\documentclass[english]{article}
\usepackage[utf8,latin9]{inputenc}
\usepackage{arxiv}
\usepackage{hyperref}       
\usepackage{url}            
\usepackage{amsmath,amssymb,amsthm,graphicx}
\usepackage{booktabs}       
\usepackage{amsfonts}       
\usepackage{nicefrac}       
\usepackage{microtype}      
\usepackage{dsfont}
\usepackage{color}
\usepackage[T1]{fontenc}

\usepackage{multirow}
\usepackage{amstext}
\usepackage[authoryear]{natbib}

\makeatletter

\providecommand{\tabularnewline}{\\}

\makeatother

\usepackage{babel}

\newtheoremstyle{thm}
{9pt}
{9pt}
{\itshape}
{}
{\bfseries}
{.}
{ }
{}
\theoremstyle{thm}
\newtheorem{theorem}{Theorem}[section]

\newtheoremstyle{def}
{9pt}
{9pt}
{}
{}
{\bfseries}
{.}
{ }
{}
\theoremstyle{def}

\newcommand{\R}{\mathbb{R}} 
\newcommand{\E}{\mathbb{E}} 
\newcommand{\PP}{\mathbb{P}} 

    \def\cd{\stackrel{\mathcal{D}}{\longrightarrow}}
    \def\cp{\stackrel{\mathcal{\PP}}{\longrightarrow}}
\renewcommand{\footnoterule}{%
	\kern -3.5pt
	\hrule width \textwidth height 1pt
	\kern 3.5pt
}

\renewcommand{\footnoterule}{%
	\kern -3.5pt
	\hrule width \textwidth height 1pt
	\kern 3.5pt
}

\makeatletter
\def\blfootnote{\xdef\@thefnmark{}\@footnotetext}
\makeatother

\title{Logistic or not logistic?}
\author{James S. Allison\\
School of Mathematical and Statistical Sciences,\\ North-West University,\\ South Africa. \\
\href{mailto:James.Allison@nwu.ac.za}{James.Allison@nwu.ac.za}\\
\And
Bruno Ebner\\
 Institute of Stochastics, \\
Karlsruhe Institute of Technology (KIT), \\
Englerstr. 2, 76133 Karlsruhe, \\
Germany\\
\href{mailto:Bruno.Ebner@kit.edu}{Bruno.Ebner@kit.edu}\\
\And
Marius Smuts\\
School of Mathematical and Statistical Sciences,\\ North-West University,\\ South Africa. \\
\href{mailto:Smuts.Marius@nwu.ac.za}{Smuts.Marius@nwu.ac.za}\\
}
\date{\today}

\begin{document}

\maketitle

\blfootnote{ {\em MSC 2010 subject
classifications.} Primary 62G10 Secondary 62E10}
\blfootnote{
{\em Key words and phrases} Goodness-of-fit; Logistic distribution; empirical characteristic function; density approach; Stein's method}

\begin{abstract}
We propose a new class of goodness-of-fit tests for the logistic distribution based on a characterisation related to the density approach in the context of Stein's method. This characterisation based test is a first of its kind for the logistic distribution. The asymptotic null distribution of the test statistic is derived and it is shown that the test is consistent against fixed alternatives. The finite sample power performance of the newly proposed
class of tests is compared to various existing tests by means of a Monte Carlo study. It is found that this new class of tests are especially powerful when the alternative distributions
are heavy tailed, like Student's t and Cauchy, or for skew alternatives such as the log-normal, gamma and chi-square distributions.  \end{abstract}

\section{Introduction}\label{sec:Intro}
 The logistic distribution apparently found its origin in the mid-nineteenth
century in the writings of \cite{verhulst1838notice,verhulst1845recherches}.
Since then it has been used in many different areas such as logistic
regression, logit models and neural networks. The logistic law
has become a popular choice of model in reliability theory and
survival analysis (see e.g., \citealp{kannisto1999trends}) and
lately in finance (\citealp{ahmad}). The United States Chess Federation
and FIDE have recently changed its formula for calculating chess ratings
of players by using the more heavy tailed logistic distribution instead
of the lighter tailed normal distribution (\citealp{A:2017,E:1978}). For a detail
account on the history and application of the logistic distribution,
the interested reader is referred to \citet{JKB:1995}.

In the literature some goodness-of-fit tests for assessing whether
the observed data are realisations from the logistic distribution
have been developed and studied. These include tests based on the
empirical distribution function (\citealp{stephens1979tests}), normalised
spacings (\citealp{lockhart1986tests}), chi-squared type statistics
(\citealp{aguirre1994chi}), orthogonal expansions (\citealp{cuadras2000some}),
empirical characteristic and moment generating functions (\citealp{meintanis2004goodness,E:2005})
and the Gini index (\citealp{alizadeh2017gini}).

\citet{balakrishnan1991handbook} provides an excellent discussion
on the logistic distribution including some of the goodness-of-fit
tests mentioned above. \citet{nikitin2019goodness} recently proposed
a test based on a characterisation of the logistic distribution involving
independent shifts. In this paper the authors remarked that ``no
goodness-of-fit tests of the composite hypothesis to the logistic
family based on characterizations are yet known''. Although, as mentioned,
some tests exists for the logistic distribution, they are few in number
compared to those for other distributions such as the normal, exponential
and the Rayleigh distribution. In this paper we propose a new class
of tests for the logistic distribution based on a new characterisation filling the gap reported by \citet{nikitin2019goodness}. To be precise, we write shorthand L$(\mu,\sigma)$, $\mu \in\R$, $\sigma > 0$, for the logistic distribution with location parameter $\mu$ and scale parameter $\sigma$ if the density is defined by
\begin{equation}\label{eq:density}
f(x,\mu,\sigma) = \frac1\sigma\frac{\exp\left(-\frac{x-\mu}{\sigma}\right)}{\left(1+\exp\left(-\frac{x-\mu}{\sigma}\right)\right)^2} = \frac1{4\sigma}\left(\mbox{sech}\left(-\frac{x-\mu}{2\sigma}\right)\right)^2, x\in\R,
\end{equation}
where $\mbox{sech}(\cdot)=(\cosh(\cdot))^{-1}$ is the hyperbolic secant. Note that $X \sim$ L$(\mu, \sigma)$ if, and only if, $\frac{X-\mu}{\sigma} \sim$ L$(0,1)$ and hence the logistic distribution belongs to the location-scale family of distributions, for a detailed discussion see \cite{JKB:1995}, chapter 23. In the following we denote the family of logistic distributions by $\mathcal{L}:=\{\mbox{L}(\mu,\sigma):\,\mu\in\R,\sigma>0\}$, a family of distributions which is closed under translation and rescaling. Let $X, X_1, X_2, \dotso$ be real-valued independent and identically distributed (iid.) random variables with distribution $\mathbb{P}^X$ defined on an underlying probability space $(\Omega,\mathcal{A},\mathbb{P})$. We test the composite hypothesis
\begin{equation}\label{eq:H0}
H_0:\;\mathbb{P}^X\in\mathcal{L}
\end{equation}
against general alternatives based on the sample $X_1,\ldots,X_n$.

The novel procedure is based on the following new characterisation of the standard Logistic distribution, which is related to the density method in the broad theory of Stein's method for distributional approximation, see for example \citet{CGS:2011,LS:2013}, and the Stein-Tikhomirov approach, see \citet{AMPS:2017}.

\begin{theorem}\label{thm:chr}
Let $X$ be a random variable with absolutely continuous density $p$ and $\E\left[ \left|\frac{1-\exp(-X)}{1+\exp(-X)}\right|\right] < \infty$. Then $X$ follows a standard logistic distribution $\mbox{\rm L}(0,1)$ if, and only if
\begin{equation}\label{eq:meanchar}
\E \left[f_t'(X)-\frac{1-\exp(-X)}{1+\exp(-X)}f_t(X)\right]= 0
\end{equation}
holds for all $t\in \R$, where $f_t(x)=\exp(itx)$ and $i$ is the imaginary unit.
\end{theorem}
\begin{proof}
For $X\sim\mbox{\rm L}(0,1)$ direct calculation shows the assertion. Let $X$ be a random variable with absolutely continuous density function $p$ such that
\begin{equation*}
\E \left[\left(it-\frac{1-\exp(-X)}{1+\exp(-X)}\right)\exp(itX)\right]= 0
\end{equation*}
holds for all $t\in \R$. Note that since $-it\E[\exp(itX)]$ is the Fourier-Stieltjes transform of the derivative of $p$ we have
\begin{equation*}
0=\E \left[\left(it-\frac{1-\exp(-X)}{1+\exp(-X)}\right)\exp(itX)\right]=\int_{-\infty}^\infty \left(-p'(x)-\frac{1-\exp(-X)}{1+\exp(-X)}p(x)\right)\exp(itx)\,\mbox{d}x
\end{equation*}
for all $t\in\R$. By standard properties of the Fourier-Stieltjes transform, we hence note that $p$ must satisfy the ordinary differential equation
\begin{equation*}
p'(x)+\frac{1-\exp(-x)}{1+\exp(-x)}p(x)=0
\end{equation*}
for all $x\in\R$. By separation of variables it is straightforward to see, that the only solution satisfying $\int_{-\infty}^\infty p(x)\mbox{d}x=1$ is $p(x)=f(x,0,1)$, $x\in\R$, and $X\sim\mbox{\rm L}(0,1)$ follows.
\end{proof}

To model the standardisation assumption, we consider the so called scaled residuals $Y_{n,1},...,Y_{n,n}$, given by
\begin{equation*}
Y_{n,j}=\frac{X_j-\widehat{\mu}_n}{\widehat{\sigma}_n}, \quad j=1,\ldots,n.
\end{equation*}
Here, $\widehat{\mu}_n=\widehat{\mu}_n(X_1,...,X_n)$ and $\widehat{\sigma}_n=\widehat{\sigma}_n(X_1,...,X_n)$ denote consistent estimators of $\mu\in\R$ and $\sigma>0$ such that
\begin{eqnarray}
\label{eq:est.mu}
\widehat{\mu}_n(bX_1+c,...,bX_n+c)&=&b\widehat{\mu}_n(X_1,...,X_n)+c,\\\label{eq:est.sigma}
\widehat{\sigma}_n(bX_1+c,...,bX_n+c)&=&b\widehat{\sigma}_n(X_1,...,X_n),
\end{eqnarray}
holds for each $b>0$ and $c\in\R$. By (\ref{eq:est.mu}) and (\ref{eq:est.sigma}) it is easy to see that $Y_{n,j}$, $j=1,\ldots,n,$ do not depend on the location nor the scale parameter, so we assume $\mu=0$ and $\sigma=1$ in the following. The test statistic
\begin{equation*}
\label{eq:statistic}
T_n=n\int_{-\infty}^\infty \left|\frac1n\sum_{j=1}^n\left(it-\frac{1-\exp(-Y_{n,j})}{1+\exp(-Y_{n,j})}\right)\exp(itY_{n,j})\right|^2\omega(t)\,\mbox{d}t
\end{equation*}
is the weighted $L^2$-distance from (\ref{eq:meanchar}) to the 0-function. Here, $\omega(\cdot)$ denotes a symmetric, positive weight function satisfying $\int_{-\infty}^\infty\omega(t)\,\mbox{d}t<\infty$, that guaranties that the considered integrals are finite. Since under the hypothesis \eqref{eq:H0} $T_n$ should be close to 0, we reject $H_0$ for \textit{large} values of $T_n$.

Note that $T_n$ is in the structural spirit of Section 5.4.2 in \citet{anastasiou2021steins}. It only depends on the scaled residuals $Y_{n,j}$, $j=1,\ldots,n$, and as a consequence it is invariant due to affine transformations of the data, i.e. w.r.t. transformations of the form $x\mapsto b x+c$, $b>0,c\in\R$. This is indeed a desirable property, since the family $\mathcal{L}$ is closed under affine transformations.

Direct calculations show with $\omega(t)=\omega_a(t)=\exp(-at^2)$, $t\in\R$, $a>0$, that the integration-free and numerically stable version is
\begin{eqnarray*}
T_{n,a}&=&\frac1{4a^2n}\sqrt{\frac{\pi}{a}}\sum_{j,k=1}^n\frac{\exp(-(Y_{n,j,k}^+)^2/4a)}{(\exp(Y_{n,j})+1)(\exp(Y_{n,k})+1)}\left[(4a^2+2a+(Y_{n,j,k}^-)^2)\exp(Y_{n,j,k}^++Y_{n,j,k}^\cdot/a)\right.
\\&&\left.-\exp(Y_{n,j,k}^\cdot/a)\left((4a^2+2a(2Y_{n,j,k}^--1)+(Y_{n,j,k}^-)^2)\exp(Y_{n,j})\right.\right.\\&&\left.\left.\hspace{2.5cm}+(4a^2-2a(2Y_{n,j,k}^-+1)+(Y_{n,j,k}^-)^2)\exp(Y_{n,k})\right)\right.
\\&& \left.+(4a^2+2a-(Y_{n,j,k}^-)^2)\exp(Y_{n,j,k}^\cdot/a)\right],
\end{eqnarray*}
with $Y_{n,j,k}^+=Y_{n,j}+Y_{n,k}$, $Y_{n,j,k}^-=Y_{n,j}-Y_{n,k}$, and $Y_{n,j,k}^\cdot=Y_{n,j}\cdot Y_{n,k}$ for $j,k=1,\ldots,n$. Here, $a$ is a so-called \textit{tuning parameter}, which allows some flexibility in the choice of the right test statistic $T_{n,a}$. A good choice of $a$ is suggested in Section \ref{sec:Sim results}.

The rest of the paper is organised as follows. In Section \ref{sec:Limit dist} the asymptotic behaviour of the new test
is investigated under the null and alternative distribution, respectively.
The results of a Monte Carlo study is presented in Section \ref{sec:Sim results}, while all the tests are applied to a real-world data set in Section 4. The
paper concludes in Section \ref{sec:Conclusion} with some concluding remarks and an outlook
for future research.

\section{Limit distribution under the null hypothesis and consistency\label{sec:Limit dist}}
In what follows let $X_1,X_2,\ldots$ be iid. random variables, and in view of affine invariance of $T_{n}$ we assume w.l.o.g. $X_1\sim \mbox{L}(0,1)$. A suitable setup for deriving asymptotic theory is the Hilbert space $\mathbb{H}$ of measurable, square integrable functions $\mathbb{H}=L^2(\mathbb{R},\mathcal{B}, \omega(t)\mbox{d}t)$, where $\mathcal{B}$ is the Borel-$\sigma$-field of $\mathbb{R}$. Notice that the functions figuring within the integral in the definition of $T_{n}$ are $(\mathcal{A} \otimes \mathcal{B}, \mathcal{B})$-measurable random elements of $\mathbb{H}$. We denote by
\begin{equation*}
	\|f\|_{\mathbb{H}} = \left( \int_{-\infty}^\infty \big|f(t)\big|^2 \, \omega(t) \, \mathrm{d}t \right)^{1/2}, \qquad \langle f, g \rangle_{\mathbb{H}}=\int_{-\infty}^\infty f(t)g(t) \, \omega(t) \, \mathrm{d}t
\end{equation*}
the usual norm and inner product in $\mathbb{H}$.  In the following, we assume that the estimators $\widehat{\mu}_n$ and $\widehat{\sigma}_n$ allow linear representations
\begin{align}
    \sqrt{n} \widehat{\mu}_n&= \frac{1}{\sqrt{n}} \sum_{j = 1}^n \psi_1 (X_{j}) + o_{\mathbb{P}}(1),\label{eq:psi_11}\\
    \sqrt{n} (\widehat{\sigma}_n - 1) &= \frac{1}{\sqrt{n}} \sum_{j = 1}^n \psi_2 (X_{j}) + o_{\mathbb{P}}(1),\label{eq:psi_21}
\end{align}
where $o_\mathbb{P}(1)$ denotes a term that converges to 0 in probability, and $\psi_1$ und $\psi_2$ are measurable functions with
\begin{equation*}
    \mathbb{E}[\psi_1 (X_{1})] = \mathbb{E}[\psi_2 (X_{1})] = 0, \quad\mbox{and}\quad \mathbb{E}[\psi_1^2 (X_{1})] < \infty,\;  \mathbb{E}[\psi_2^2 (X_{1})] < \infty. \label{eq:psi2}
\end{equation*}
The interested reader finds formulas for the functions $\psi_1$ and $\psi_2$ in Appendix \ref{app:linrep} for maximum-likelihood and moment estimators.
By the symmetry of the weight function $\omega(\cdot)$ straightforward calculations show
\begin{equation*}
    T_n=\int_{-\infty}^\infty Z_n^2(t)\,\omega(t)\mbox{d}t,
\end{equation*}
where
\begin{equation*}
Z_n(t)=\frac1{\sqrt{n}}\sum_{j=1}^n\kappa(t,Y_{n,j}),\quad t\in\R,
\end{equation*}
and
\begin{equation*}
\kappa(t,x)=(1+\exp(-x))^{-1}\Big[\big((1-t)\cos(xt)-(t+1)\sin(xt)\big)\exp(-x)-(t+1)\cos(xt)-(t-1)\sin(xt)\Big],\quad t,x\in\R.
\end{equation*}
Clearly, $Z_n(t)$ is a sum of dependent random variables. In order to find an asymptotic equivalent stochastic process we use a first order multivariate Taylor expansion and consider with
\begin{eqnarray*}
h(t,x)&=&\left( 1+{\exp(-x)} \right) ^{-2}\Big[\left( \left( t+1 \right)\cos \left( xt \right)   -\left( t-1 \right)\sin
 \left( xt \right)    \right) t \exp(-2x)
\\&&+2\, \left( {t}^{2}+1 \right)  \left( \cos \left( xt
 \right) -\sin \left( xt \right)  \right) \exp(-x)+ \left( \left( t-1 \right)\cos
 \left( xt \right)   -\left( t+1 \right) \sin \left( xt \right)
  \right) t
\Big],\quad t,x\in\R,
\end{eqnarray*}
the helping process
\begin{eqnarray*}
Z_n^*(t)&=&\frac1{\sqrt{n}}\sum_{j=1}^n\kappa(t,X_j)+\widehat{\mu}_nh(t,X_j)+(\widehat{\sigma}_n-1)X_jh(t,X_j),\quad t\in\R.
\end{eqnarray*}
In view of \eqref{eq:psi_11} and \eqref{eq:psi_21} we define the second helping process
\begin{eqnarray*}
Z_n^{**}(t)&=&\frac1{\sqrt{n}}\sum_{j=1}^n\kappa(t,X_j)+\psi_1 (X_{j})\E[h(t,X_1)]+\psi_2 (X_{j})\E[X_1h(t,X_1)],\quad t\in\R,
\end{eqnarray*}
which is a sum of centered iid. random variables. Note that using
\begin{eqnarray*}
\E(\exp(-2X_1)/(1+\exp(-X_1))^2)=1/3, && \E(|X_1|\exp(-X_1)/(1+\exp(-X_1))^2)=\log(2)/3-1/12,\\ \E(\exp(-X_1)/(1+\exp(-X_1))^2)=1/6,  &\mbox{and}& \E(|X_1|\exp(-2X_1)/(1+\exp(-X_1))^2)=2\log(2)/3+1/12,
\end{eqnarray*}
we have by straightforward calculations $\E[|h(t,X_1)|]<\infty$ and $\E[|X_1h(t,X_1)|]<\infty$. In the following, we denote by $\cd$ weak convergence (or alternatively convergence in distribution), whenever random elements (or random variables) are considered, and in the same manner by $\cp$ convergence in probability.
\begin{theorem}
Under the standing assumptions, we have
\begin{equation*}
    Z_n\cd Z,\quad\mbox{as}\;n\rightarrow\infty,
\end{equation*}
in $\mathbb{H}$, where $Z$ is a centred Gaussian process having covariance kernel
\begin{eqnarray*}\label{eq:covker}
    K(s,t)&=&\mathbb{E}[\kappa(s,X_1)\kappa(t,X_1)]+\E[h(s,X_1)]\E[\psi_1(X_1)\kappa(t,X_1)]+\E[h(t,X_1)]\E[\psi_1(X_1)\kappa(s,X_1)]\\\nonumber
    &&+\E[X_1h(s,X_1)]\E[\psi_2(X_1)\kappa(t,X_1)]+\E[X_1h(t,X_1)]\E(\psi_2(X_1)\kappa(s,X_1)]\\\nonumber
    &&+\E[\psi_1^2(X_1)]\E[h(s,X_1)]\E[h(t,X_1)]+\E[\psi_2^2(X_1)]\E[X_1h(s,X_1)]\E[X_1h(t,X_1)]\\\nonumber
    &&+\E[\psi_1(X_1)\psi_2(X_1)]\left(\E[h(s,X_1)]\E[X_1h(t,X_1)]+\E[X_1h(s,X_1)]\E[h(t,X_1)]\right),\quad s,t\in\R.
\end{eqnarray*}
Furthermore, we have $T_n\cd \|Z\|^2_{\mathbb{H}}$, as $n\rightarrow\infty$.
\end{theorem}
\begin{proof}
In a first step, we note that after some algebra using a multivariate Taylor expansion around $(\mu,\sigma)=(0,1)$ we have
\begin{equation*}
    \|Z_n-Z_n^*\|_{\mathbb{H}}\cp 0\quad\mbox{as}\;n\rightarrow\infty.
\end{equation*}
Furthermore using the linear representations in \eqref{eq:psi_11} and \eqref{eq:psi_21} and the law of large numbers in Hilbert spaces, it follows that
\begin{equation*}
    \|Z_n^*-Z_n^{**}\|_{\mathbb{H}}\cp 0\quad\mbox{as}\;n\rightarrow\infty,
\end{equation*}
and by the triangular equation, we see that $Z_n$ has the same limiting distribution as $Z_n^{**}$. Since by the central limit theorem in Hilbert spaces $Z_n^{**}\cd Z$ in $\mathbb{H}$, where $Z$ is the stated Gaussian limit process with covariance kernel $K(s,t)=\mathbb{E}[Z_1^{**}(s)Z_1^{**}(t)]$ which gives the stated formula after a short calculation. The next statement is a direct consequence of the continuous mapping theorem.
\end{proof}

In the rest of this section we assume that the underlying distribution is a fixed alternative to $H_0$ and that the distribution is absolutely continuous, as well as in view of affine invariance of the test statistic, we assume  $\E[X]=0$ and $\E[X^2]<\infty$. Furthermore, we assume that
\begin{equation*}\label{eq:est_conv}
    (\widehat{\mu}_n,\widehat{\sigma}_n)\cp(0,1), \quad\mbox{as}\;n\rightarrow\infty.
\end{equation*}

\begin{theorem}\label{thm:contalt} Under the standing assumptions, we have as $n\rightarrow\infty$,
\begin{equation*}
\frac{T_n}{n}\cp \int_{-\infty}^\infty \left|\E\left[\left(iX-\frac{1-\exp(-X)}{1+\exp(-X)}\right)\exp(itX)\right]\right|^2\omega(t)\mbox{d}t=\Delta.
\end{equation*}
\end{theorem}
The proof of Theorem \ref{thm:contalt} follows the lines of the proof of Theorem 3.1 in \cite{EEK:2021} and since it does not provide further insights it is omitted. Notice that by the characterisation of the logistic law in Theorem \ref{thm:chr}, we have $\Delta=0$ if and only if $X\sim\mbox{L}(0,1)$. This implies that $T_n\cp\infty$, as $n\rightarrow\infty$, for any alternative with existing second moment. Thus we conclude that the test based on $T_n$ is consistent against each such alternative.

\section{Simulation results\label{sec:Sim results}}

In this section the finite sample performance of the newly proposed
test $T_{n,a}$ is compared to various existing tests for the logistic
distribution by means of a Monte Carlo study. We consider the traditional
tests (based on the empirical distribution function) of Kolmogorov-Smirnov
($\text{KS}_{n}$), Cram\'{e}r-von Mises ($\text{CM}_{n}$), Anderson-Darling
($\text{AD}_{n}$) and Watson ($\text{WA}_{n}$), a test proposed by \cite{alizadeh2017gini} based on
an estimate of the Gini index ($\text{G}_{n}$), as well as a test
by \cite{meintanis2004goodness}, based on the empirical characteristic function, with calculable
form
\[
\text{R}_{n,v}=\frac{4v^{2}\pi^{2}}{n}\sum_{j,k=1}^{n}\frac{\text{sinh}\left(Y_{j}+Y_{k}\right)}{\left(Y_{j}+Y_{k}\right)\left[4v^{2}\pi^{2}+(Y_{j}+Y_{k})^{2}\right]}-4\pi^{2}\sum_{j=1}^{n}S\left(v,Y_{j}\right)+n\left[\frac{2v\pi^{2}}{3}+2\sum_{k=1}^{v-1}\frac{v-k}{k^{2}}\right],
\]

where
\[
S\left(v,x\right)=\sum_{k=1}^{v}\frac{\left(2k-1\right)\left[\left\{ x^{2}+\left(2k-1\right)^{2}\pi^{2}\right\} \text{cosh}\left(x\right)-2x\text{sinh}\left(x\right)\right]}{\left[x^{2}+\left(2k-1\right)^{2}\pi^{2}\right]^{2}}.
\]

We also include a new test ($S_{n}$) constructed similarly to that
of $T_{n,a}$, but setting $f_t(x)=\text{exp}(tx),\ t\in(-1,1)$, in
Theorem 1.1. This new ``moment generating'' function based test is then given by the $L^2$-statistic
\[
S_{n}=\frac{n}{2}\int_{-1}^1 \left|\frac1n\sum_{j=1}^n\left(t-\frac{1-\exp(-Y_{n,j})}{1+\exp(-Y_{n,j})}\right)\exp(tY_{n,j})\right|^2\,\mbox{d}t.
\]
Direct calculations lead with $Y_{n,j,k}^+=Y_{n,j}+Y_{n,k}$ to the numerical stable version
\begin{eqnarray*}
S_n\hspace{-2mm}&=&\hspace{-2mm}\frac1n\sum_{j,k=1}^n \frac{(e^{Y_{n,j}}+e^{Y_{n,k}}+e^{Y_{n,j,k}^+}+1)^{-1}}{(Y_{n,j,k}^+)^{3}}\hspace{-1mm}\left(e^{2Y_{n,j,k}^+}-e^{-Y_{n,j,k}^+}+(1-Y_{n,j,k}^+)\left(e^{2Y_{n,j}+Y_{n,k}}+e^{Y_{n,j}+2Y_{n,k}}\right)\right.\\
&&\left.+e^{Y_{n,j,k}^+}(2(Y_{n,j,k}^+)^2-2Y_{n,j,k}^++1)-(1+Y_{n,j,k}^+)(e^{-Y_{n,j}}+e^{-Y_{n,k}})-2(Y_{n,j,k}^+)^2-2Y_{n,j,k}^+-1\right).
\end{eqnarray*}

A significance level of $\alpha = 0.05$ was used throughout the study and empirical
critical values were obtained from $100\ 000$ independent Monte Carlo
replications, with unknown parameters estimated by method of moments
(maximum likelihood estimation yielded similar results, therefore
we only display results based on method of moments). The critical values for $T_{n,a}$ are given in Table \ref{Critical values} for different values of $\alpha$, $n$ and $a$. The power estimates
were calculated for sample sizes $n=20$ and $n=50$ using $10\ 000$
independent Monte Carlo simulations. The alternative distributions
considered were the Normal ($\text{N}),$Student's t ($t$), Cauchy
($\text{C}$), Laplace ($\text{LP}),$log-normal ($\text{LN}),$gamma
($\Gamma$), uniform ($\text{U}$), beta ($\text{B}$) and chi-square
($\chi^{2}$) distributions. Tables \ref{Powers20} and \ref{Powers 50}
contain these power estimates. Apart from the mentioned alternative
distributions, we also considered some local alternatives. Table \ref{Local Cauchy}
contains the local power estimates for $n=20$ (top row) and $n=50$
(bottom row), where we simulated data from a mixture of the logistic
and Cauchy distribution, i.e. we sample from a logistic distribution
with probability $1-p$ and from a Cauchy distribution with probability
$p$. Similarly, Table \ref{local LN} contains the local power estimates
where we simulated data from a mixture of the logistic and log-normal
distribution. The tables contain the percentage of times that
the null hypothesis in (1) is rejected, rounded to the nearest integer. All calculations were performed in R \citep{rteam}.

\begin{table}
\begin{center}
\caption{Critical values for the new test $T_{n,a}$}

\begin{centering}
\medskip{}
\par\end{centering}

\begin{tabular}{|l|lll|lll|}

\hline
\multirow{2}{*}{} & \multicolumn{3}{c|}{$n=20$} & \multicolumn{3}{c|}{$n=50$}\tabularnewline
\cline{2-7} \cline{3-7} \cline{4-7} \cline{5-7} \cline{6-7} \cline{7-7}
 & $\text{T}_{n,3}$ & $\text{T}_{n,4}$ & $\text{T}_{n,5}$ & $\text{T}_{n,3}$ & $\text{T}_{n,4}$ & $\text{T}_{n,5}$\tabularnewline
\hline
$\alpha=0.01$ & 1.011 & 0.701 & 0.525 & 1.091 & 0.759 & 0.580\tabularnewline
$\alpha=0.05$ & 0.684 & 0.459 & 0.339 & 0.714 & 0.487 & 0.363\tabularnewline
$\alpha=0.10$ & 0.531 & 0.350 & 0.254 & 0.555 & 0.374 & 0.276\tabularnewline
\hline
\end{tabular}\label{Critical values}
\end{center}
\end{table}

\begin{table}
\caption{Estimated powers for alternative distributions for $n=20$}

\begin{centering}
\medskip{}
\par\end{centering}
\begin{tabular}{|l|cccc|ccc|ccccc|}
\hline
Alternative & $\text{T}_{n,3}$ & $\text{T}_{n,4}$ & $\text{T}_{n,5}$ & $\text{S}_{n}$ & $\text{R}_{n,1}$ & $\text{R}_{n,2}$ & $\text{R}_{n,3}$ & $\text{KS}_{n}$ & $\text{CM}_{n}$ & $\text{AD}_{n}$ & $\text{WA}_{n}$ & $\text{G}_{n}$\tabularnewline
\hline
$\text{L}(0,1)$ & 5 & 5 & 5 & 5 & 5 & 5 & 5 & 5 & 5 & 5 & 5 & 5\tabularnewline
$\text{N}(0,1)$ & 3 & 3 & 2 & 1 & 6 & \textbf{7} & \textbf{9} & 4 & 4 & 4 & 4 & 5\tabularnewline
$t_{2}$ & \textbf{37} & \textbf{37} & \textbf{37} & \textbf{38} & 31 & 27 & 23 & 33 & \textbf{37} & \textbf{37} & 36 & 33\tabularnewline
$t_{5}$ & \textbf{9} & \textbf{9} & \textbf{9} & \textbf{10} & \textbf{9} & 8 & 7 & 7 & 8 & 8 & 8 & 8\tabularnewline
$t_{10}$ & 4 & 4 & 4 & 4 & 5 & \textbf{6} & \textbf{6} & 5 & 4 & 5 & 5 & 5\tabularnewline
\textbf{$\text{C}(0,1)$} & 76 & 75 & 74 & 69 & 62 & 58 & 53 & 76 & \textbf{79} & \textbf{79} & \textbf{79} & 75\tabularnewline
$\text{LP}(0,1)$ & \textbf{13} & \textbf{13} & \textbf{13} & \textbf{13} & 8 & 6 & 5 & 12 & \textbf{13} & \textbf{13} & \textbf{13} & 10\tabularnewline
$\text{LN}(1)$ & \textbf{87} & \textbf{87} & \textbf{87} & 76 & 48 & 39 & 33 & 75 & 85 & \textbf{87} & 80 & 28\tabularnewline
$\text{LN}(1.5)$ & \textbf{98} & \textbf{98} & \textbf{98} & 94 & 70 & 61 & 53 & 94 & 97 & \textbf{98} & 97 & 52\tabularnewline
$\text{LN}(2)$ & 99 & 99 & 99 & 98 & 81 & 73 & 64 & 98 & \textbf{100} & \textbf{100} & 99 & 72\tabularnewline
$\Gamma(1)$ & \textbf{70} & \textbf{70} & 69 & 52 & 26 & 20 & 16 & 53 & 67 & \textbf{71} & 61 & 17\tabularnewline
$\Gamma(2)$ & \textbf{41} & \textbf{41} & 40 & 28 & 17 & 14 & 12 & 28 & 36 & 38 & 31 & 15\tabularnewline
$\Gamma(3)$ & \textbf{26} & \textbf{26} & \textbf{26} & 17 & 12 & 11 & 11 & 19 & 23 & 24 & 20 & 13\tabularnewline
$\text{U}(\sqrt{3},-\sqrt{3})$ & 16 & 8 & 5 & 0 & 48 & \textbf{55} & \textbf{58} & 13 & 21 & 28 & 27 & 30\tabularnewline
$\text{B}(2,2)$ & 5 & 3 & 2 & 0 & 19 & \textbf{24} & \textbf{28} & 6 & 9 & 11 & 12 & 14\tabularnewline
$\text{B}(3,5)$ & 13 & 11 & 10 & 4 & 11 & 13 & \textbf{14} & 11 & 13 & \textbf{14} & 13 & 13\tabularnewline
$\chi_{2}^{2}$ & \textbf{71} & \textbf{71} & 70 & 52 & 27 & 21 & 17 & 53 & 67 & \textbf{71} & 61 & 16\tabularnewline
$\chi_{5}^{2}$ & \textbf{32} & \textbf{32} & 31 & 21 & 15 & 13 & 12 & 23 & 27 & 29 & 24 & 14\tabularnewline
$\chi_{10}^{2}$ & \textbf{15} & \textbf{15} & \textbf{15} & 10 & 10 & 10 & 10 & 12 & 14 & \textbf{15} & 12 & 11\tabularnewline
$\chi_{15}^{2}$ & \textbf{12} & \textbf{11} & \textbf{11} & 7 & 8 & 9 & 10 & 9 & 10 & \textbf{11} & 10 & 9\tabularnewline
\hline
\end{tabular}\label{Powers20}
\end{table}

\begin{table}
\caption{Estimated powers for alternative distributions for $n=50$}

\begin{centering}
\medskip{}
\par\end{centering}
\begin{tabular}{|l|cccc|ccc|ccccc|}
\hline
Alternative & $\text{T}_{n,3}$ & $\text{T}_{n,4}$ & $\text{T}_{n,5}$ & $\text{S}_{n}$ & $R_{n,1}$ & $R_{n,2}$ & $R_{n,3}$ & $\text{KS}_{n}$ & $\text{CM}_{n}$ & $\text{AD}_{n}$ & $\text{WA}_{n}$ & $\text{G}_{n}$\tabularnewline
\hline
$\text{L}(0,1)$ & 5 & 5 & 5 & 5 & 5 & 5 & 5 & 5 & 5 & 5 & 5 & 5\tabularnewline
$\text{N}(0,1)$ & 5 & 3 & 3 & 0 & 4 & 7 & \textbf{10} & 5 & 6 & 6 & 7 & \textbf{9}\tabularnewline
$t_{2}$ & 65 & 64 & 64 & 59 & 57 & 54 & 51 & 60 & 66 & \textbf{67} & \textbf{67} & 65\tabularnewline
$t_{5}$ & 11 & 12 & 12 & \textbf{15} & \textbf{14} & 13 & 12 & 9 & 11 & 11 & 10 & 11\tabularnewline
$t_{10}$ & 5 & 4 & 4 & 4 & 5 & \textbf{6} & \textbf{6} & 5 & 5 & 5 & \textbf{6} & 5\tabularnewline
$\text{C}(0,1)$ & 98 & 98 & 97 & 91 & 91 & 90 & 88 & 98 & \textbf{99} & \textbf{99} & \textbf{99} & 98\tabularnewline
$\text{LP}(0,1)$ & 19 & 18 & 18 & 15 & 13 & 11 & 9 & 20 & \textbf{23} & 22 & \textbf{24} & 19\tabularnewline
$\text{LN}(1)$ & \textbf{100} & \textbf{100} & \textbf{100} & 96 & 83 & 76 & 70 & 99 & \textbf{100} & \textbf{100} & \textbf{100} & 48\tabularnewline
$\text{LN}(1.5)$ & \textbf{100} & \textbf{100} & \textbf{100} & \textbf{100} & 96 & 93 & 90 & \textbf{100} & \textbf{100} & \textbf{100} & \textbf{100} & 87\tabularnewline
$\text{LN}(2)$ & \textbf{100} & \textbf{100} & \textbf{100} & \textbf{100} & 99 & 98 & 96 & \textbf{100} & \textbf{100} & \textbf{100} & \textbf{100} & 98\tabularnewline
$\Gamma(1)$ & \textbf{99} & \textbf{99} & \textbf{99} & 73 & 48 & 39 & 31 & 94 & \textbf{99} & \textbf{99} & 97 & 21\tabularnewline
$\Gamma(2)$ & \textbf{87} & \textbf{87} & \textbf{87} & 39 & 27 & 22 & 19 & 67 & 81 & 86 & 73 & 28\tabularnewline
$\Gamma(3)$ & 68 & \textbf{69} & \textbf{69} & 24 & 18 & 17 & 14 & 48 & 59 & 65 & 50 & 30\tabularnewline
$\text{U}(\sqrt{3},-\sqrt{3})$ & 78 & 66 & 51 & 0 & 93 & \textbf{97} & \textbf{98} & 44 & 68 & 84 & 77 & 74\tabularnewline
$\text{B}(2,2)$ & 29 & 19 & 11 & 0 & 47 & \textbf{64} & \textbf{71} & 16 & 26 & 35 & 34 & 42\tabularnewline
$\text{B}(3,5)$ & \textbf{46} & 43 & 40 & 1 & 12 & 18 & 21 & 30 & 39 & \textbf{45} & 37 & 40\tabularnewline
$\chi_{2}^{2}$ & \textbf{99} & \textbf{99} & \textbf{99} & 75 & 49 & 40 & 33 & 95 & \textbf{99} & \textbf{99} & 97 & 20\tabularnewline
$\chi_{5}^{2}$ & 78 & \textbf{79} & \textbf{79} & 30 & 22 & 19 & 16 & 57 & 70 & 75 & 61 & 30\tabularnewline
$\chi_{10}^{2}$ & \textbf{46} & \textbf{46} & \textbf{46} & 13 & 12 & 13 & 13 & 31 & 38 & 41 & 32 & 29\tabularnewline
$\chi_{15}^{2}$ & \textbf{31} & \textbf{31} & 30 & 8 & 8 & 10 & 11 & 22 & 26 & 28 & 22 & 24\tabularnewline
\hline
\end{tabular}\label{Powers 50}
\end{table}

\begin{table}
\caption{Estimated local powers for mixture with Cauchy distribution for $n=20$
(top row) and $n=50$ (bottom row)}

\begin{centering}
\medskip{}
\par\end{centering}
\begin{tabular}{|c|cccc|cccc|cccc|}
\hline
$\text{Mixing proportion} (p)$ & $\text{T}_{n,3}$ & $\text{T}_{n,4}$ & $\text{T}_{n,5}$ & $\text{S}_{n}$ & $R_{n,1}$ & $R_{n,2}$ & $R_{n,3}$ & $\text{KS}_{n}$ & $\text{CM}_{n}$ & $\text{AD}_{n}$ & $\text{WA}_{n}$ & $\text{G}_{n}$\tabularnewline
\hline
\multirow{2}{*}{$0$} & 5 & 5 & 5 & 5 & 5 & 5 & 5 & 5 & 5 & 5 & 5 & 5\tabularnewline
 & 5 & 5 & 5 & 5 & 5 & 5 & 5 & 5 & 5 & 5 & 5 & 5\tabularnewline
\hline
\multirow{2}{*}{$0.05$} & \textbf{12} & \textbf{12} & \textbf{12} & \textbf{12} & \textbf{12} & 11 & 11 & 11 & \textbf{12} & \textbf{12} & 11 & 11\tabularnewline
 & 20 & 20 & 20 & \textbf{22} & \textbf{22} & \textbf{22} & 21 & 18 & 19 & 19 & 18 & 18\tabularnewline
\hline
\multirow{2}{*}{$0.1$} & 17 & \textbf{18} & \textbf{18} & \textbf{19} & 17 & 16 & 15 & 16 & 17 & 17 & 16 & 16\tabularnewline
 & 30 & 31 & 31 & \textbf{35} & \textbf{34} & 33 & 32 & 28 & 30 & 31 & 30 & 31\tabularnewline
\hline
\multirow{2}{*}{$0.15$} & \textbf{24} & \textbf{24} & \textbf{24} & \textbf{25} & 23 & 21 & 20 & 21 & 23 & 23 & 22 & 22\tabularnewline
 & 41 & 42 & 43 & \textbf{46} & \textbf{45} & 44 & 43 & 38 & 40 & 41 & 40 & 43\tabularnewline
\hline
\multirow{2}{*}{$0.2$} & 28 & \textbf{29} & \textbf{29} & \textbf{31} & 28 & 26 & 24 & 25 & 27 & 27 & 26 & 27\tabularnewline
 & 51 & 52 & 52 & \textbf{55} & \textbf{54} & 53 & 51 & 48 & 51 & 51 & 50 & 51\tabularnewline
\hline
\multirow{2}{*}{$0.3$} & \textbf{38} & \textbf{38} & \textbf{38} & \textbf{39} & 36 & 33 & 31 & 34 & 36 & 37 & 35 & 37\tabularnewline
 & 66 & 66 & 66 & \textbf{68} & \textbf{67} & 65 & 64 & 62 & 65 & 66 & 65 & 65\tabularnewline
\hline
\multirow{2}{*}{$0.4$} & 46 & \textbf{47} & \textbf{47} & \textbf{47} & 42 & 39 & 36 & 42 & 45 & 45 & 44 & 43\tabularnewline
 & \textbf{77} & \textbf{77} & \textbf{77} & 76 & 76 & 74 & 72 & 73 & 76 & \textbf{77} & 76 & 76\tabularnewline
\hline
\multirow{2}{*}{$0.5$} & \textbf{55} & \textbf{55} & \textbf{55} & 54 & 49 & 46 & 42 & 51 & 54 & 54 & 53 & 51\tabularnewline
 & 84 & 84 & 84 & 82 & 82 & 80 & 78 & 82 & 84 & \textbf{85} & \textbf{85} & 83\tabularnewline
\hline
\multirow{2}{*}{$0.6$} & \textbf{60} & \textbf{60} & \textbf{60} & 58 & 52 & 48 & 45 & 57 & \textbf{60} & \textbf{60} & 59 & 57\tabularnewline
 & 89 & 89 & 89 & 86 & 85 & 84 & 82 & 87 & \textbf{90} & \textbf{90} & \textbf{90} & 89\tabularnewline
\hline
\multirow{2}{*}{$0.7$} & 65 & 65 & 65 & 63 & 56 & 52 & 48 & 63 & \textbf{66} & \textbf{66} & 65 & 62\tabularnewline
 & \textbf{93} & 92 & 92 & 88 & 87 & 86 & 84 & 91 & \textbf{93} & \textbf{93} & \textbf{94} & 92\tabularnewline
\hline
\multirow{2}{*}{$0.8$} & 70 & 69 & 69 & 65 & 58 & 54 & 50 & 68 & \textbf{71} & \textbf{71} & \textbf{71} & 67\tabularnewline
 & \textbf{96} & 95 & 95 & 90 & 89 & 88 & 86 & 95 & \textbf{96} & \textbf{96} & \textbf{96} & 95\tabularnewline
\hline
\multirow{2}{*}{$0.9$} & 73 & 73 & 72 & 68 & 61 & 57 & 52 & 72 & \textbf{76} & \textbf{75} & \textbf{75} & 71\tabularnewline
 & 97 & 97 & 96 & 91 & 91 & 90 & 88 & 96 & \textbf{98} & \textbf{98} & \textbf{98} & 97\tabularnewline
\hline
\multirow{2}{*}{$1$} & 77 & 76 & 75 & 70 & 63 & 59 & 54 & 77 & \textbf{80} & \textbf{80} & \textbf{80} & 75\tabularnewline
 & 98 & 98 & 98 & 92 & 91 & 90 & 88 & 98 & \textbf{99} & \textbf{99} & \textbf{99} & 98\tabularnewline
\hline
\end{tabular}\label{Local Cauchy}
\end{table}

\begin{table}
\caption{Estimated local powers for mixture with Log-Normal distribution for
$n=20$ (top row) and $n=50$ (bottom row)}

\begin{centering}
\medskip{}
\par\end{centering}
\begin{tabular}{|c|cccc|ccc|ccccc|}
\hline
$\text{Mixing proportion} (p)$ & $\text{T}_{n,3}$ & $\text{T}_{n,4}$ & $\text{T}_{n,5}$ & $\text{S}_{n}$ & $R_{n,1}$ & $R_{n,2}$ & $R_{n,3}$ & $\text{KS}_{n}$ & $\text{CM}_{n}$ & $\text{AD}_{n}$ & $\text{WA}_{n}$ & $\text{G}_{n}$\tabularnewline
\hline
\multirow{2}{*}{$0$} & 5 & 5 & 5 & 5 & 5 & 5 & 5 & 5 & 5 & 5 & 5 & 5\tabularnewline
 & 5 & 5 & 5 & 5 & 5 & 5 & 5 & 5 & 5 & 5 & 5 & 5\tabularnewline
 \hline
\multirow{2}{*}{$0.05$} & \textbf{7} & \textbf{7} & \textbf{7} & \textbf{7} & \textbf{7} & 6 & 6 & 6 & \textbf{7} & \textbf{7} & \textbf{7} & 5\tabularnewline
 & 8 & 8 & 8 & \textbf{9} & \textbf{9} & 8 & 8 & 7 & 7 & 8 & 7 & 6\tabularnewline
 \hline
\multirow{2}{*}{$0.1$} & \textbf{8} & \textbf{8} & \textbf{8} & \textbf{9} & \textbf{8} & 7 & 6 & 7 & \textbf{8} & \textbf{8} & 7 & 6\tabularnewline
 & 10 & 11 & 11 & \textbf{13} & \textbf{12} & 11 & 11 & 10 & 11 & 11 & 11 & 8\tabularnewline
 \hline
\multirow{2}{*}{$0.15$} & \textbf{10} & \textbf{10} & \textbf{10} & \textbf{11} & 9 & 8 & 7 & 8 & 9 & 9 & 9 & 7\tabularnewline
 & 13 & 14 & 14 & \textbf{16} & \textbf{15} & 14 & 13 & 13 & 14 & 14 & 13 & 10\tabularnewline
 \hline
\multirow{2}{*}{$0.2$} & 11 & \textbf{12} & \textbf{12} & \textbf{12} & 10 & 9 & 8 & 11 & 11 & 11 & 11 & 8\tabularnewline
 & 17 & 17 & 17 & \textbf{20} & \textbf{19} & 18 & 16 & 17 & 18 & 18 & 18 & 13\tabularnewline
 \hline
\multirow{2}{*}{$0.3$} & \textbf{15} & \textbf{15} & \textbf{15} & \textbf{16} & 13 & 11 & 9 & \textbf{15} & \textbf{15} & \textbf{15} & \textbf{15} & 10\tabularnewline
 & 23 & 24 & 24 & 25 & 24 & 23 & 21 & 26 & \textbf{27} & \textbf{27} & \textbf{28} & 19\tabularnewline
 \hline
\multirow{2}{*}{$0.4$} & 19 & 19 & 19 & \textbf{20} & 14 & 12 & 10 & 19 & \textbf{20} & 19 & 19 & 13\tabularnewline
 & 33 & 33 & 33 & 32 & 31 & 28 & 26 & 38 & \textbf{41} & 39 & \textbf{42} & 27\tabularnewline
 \hline
\multirow{2}{*}{$0.5$} & 25 & 25 & 25 & 25 & 18 & 15 & 12 & 26 & \textbf{28} & 26 & \textbf{27} & 16\tabularnewline
 & 43 & 42 & 42 & 39 & 36 & 33 & 31 & 52 & \textbf{54} & 52 & \textbf{56} & 36\tabularnewline
 \hline
\multirow{2}{*}{$0.6$} & 31 & 31 & 31 & 29 & 21 & 17 & 14 & 32 & \textbf{35} & 34 & \textbf{35} & 20\tabularnewline
 & 57 & 55 & 53 & 45 & 42 & 39 & 36 & 65 & \textbf{70} & 67 & \textbf{72} & 45\tabularnewline
 \hline
\multirow{2}{*}{$0.7$} & 41 & 40 & 40 & 36 & 26 & 21 & 17 & 41 & \textbf{46} & 44 & \textbf{46} & 25\tabularnewline
 & 71 & 69 & 67 & 53 & 49 & 45 & 42 & 78 & \textbf{83} & 81 & \textbf{84} & 51\tabularnewline
 \hline
\multirow{2}{*}{$0.8$} & 51 & 50 & 49 & 44 & 30 & 25 & 20 & 51 & \textbf{57} & \textbf{56} & \textbf{56} & 28\tabularnewline
 & 85 & 83 & 82 & 63 & 58 & 54 & 50 & 88 & \textbf{92} & \textbf{92} & \textbf{93} & 54\tabularnewline
 \hline
\multirow{2}{*}{$0.9$} & 66 & 66 & 65 & 56 & 37 & 30 & 26 & 61 & \textbf{69} & \textbf{69} & 67 & 28\tabularnewline
 & 95 & 94 & 94 & 76 & 68 & 63 & 58 & 96 & \textbf{98} & \textbf{98} & \textbf{98} & 53\tabularnewline
 \hline
\multirow{2}{*}{$1$} & \textbf{87} & \textbf{87} & \textbf{87} & 77 & 48 & 39 & 33 & 75 & 85 & \textbf{87} & 80 & 29\tabularnewline
 & \textbf{100} & \textbf{100} & \textbf{100} & 96 & 83 & 76 & 70 & 99 & \textbf{100} & \textbf{100} & \textbf{100} & 48\tabularnewline
\hline
\end{tabular}\label{local LN}
\end{table}

The newly proposed tests are especially powerful when the alternative distributions
are heavy tailed, like the Student's t and Cauchy, or for skew alternatives such as the log-normal, gamma and chi-square distributions. The
test of Meintanis produces the highest estimated powers when the alternatives
have lighter tails or have bounded support such as the uniform and
Beta distributions. When comparing the traditional tests, it is clear
that the Anderson-Darling test has superior estimated powers.

For the mixture of the logistic and Cauchy distribution, the newly
proposed tests ($T_{n,a}$ and $S_{n}$) as well as the test of Meintanis
have the highest estimated powers for small values of the mixing parameter
$p$ (i.e. closer to the null distribution). The more traditional
tests have slightly higher estimated powers for increasing values
of the mixing parameter. This trend is similar when considering the
mixture of the logistic and log-normal distributions.

Overall, the newly proposed test $T_{n,a}$ performs favourably relative
to the existing tests and to a lessor extend the test $S_{n}$, which
is based on the moment generating function. For practical implementation of the test, we advise choosing the tuning parameter as $a=3$ as this choice produced high estimated powers for most alternatives considered. Alternatively, one can use the methods described in \cite{allison2015data} or \cite{tenreiro2019automatic} to choose this parameter data dependently.

\section{Practical application \label{sec:Practical}}

In this section all the tests considered in the Monte Carlo study are applied to the 'Bladder cancer' data set. This data set contains the monthly remission times of 128 patients, denoted by $x_1,\dots,x_{128}$,  who were diagnosed with bladder cancer and can be found in \cite{lee2003statistical}. The data was also studied and analysed by \cite{Hadi2021} and \cite{al2016log}. We are interested in testing whether the log of the remission times, $w_j=\text{log}({x_j}), j = 1,\dots,128$, follow a logistic distribution. The method of moment estimates of ${\mu}$ and  ${\sigma}$ are $\hat{\mu}_n= \hat{\mu}(w_{1},\dots,w_{128}) = 1.753$ and $\hat{\sigma}_n = \hat{\sigma}(w_{1},\dots,w_{128}) = 0.592$, respectively. Figure 1 represents the probability plot of $G^{-1}(\frac{k}{n+1})$ vs $y_{(k)}$, where $G^{-1}(\cdot)$ denotes the quantile function of the standard logistic distribution and $y_{(k)}=\frac{w_{(k)}-\hat{\mu}_n}{\hat{\sigma}_n}, k = 1,\dots,128$. This probability plot suggest that the underlying distribution of the data might be the logistic distribution. Table \ref{P-values} contains the test statistic values as well as the corresponding estimated $p$-values (calculated based on 10 000 samples of size 128 simulated from the standard logistic distribution) for the 8 tests for testing the goodness-of-fit for the logistic distribution. All the tests do not reject the null hypothesis that the log of the remission times is logistic distributed. These findings are in agreement with that of \cite{al2016log}, where they concluded that the remission times follows a log-logistic distribution.

\begin{figure}
\begin{center}
\includegraphics[scale=0.7]{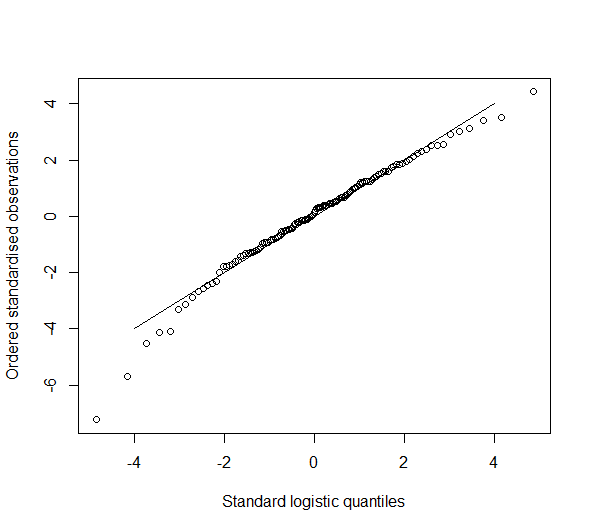}

\caption{Probability plot of the cancer data set}
\end{center}
\end{figure}\label{PP_plot}

\begin{table}
\begin{center}
\caption{Estimated $p$-values for the cancer data set}

\begin{centering}
\medskip{}
\par\end{centering}
\begin{tabular}{|c|cc|}
\hline
Test & Test statistic value & $p-$value\tabularnewline
\hline
$\text{T}_{n,3}$ & 0.500 & 0.171\tabularnewline
$\text{S}_{n}$ & 19.75 & 0.329\tabularnewline
$\text{R}_{n,1}$ & 169.4 & 0.602\tabularnewline
$\text{KS}_{n}$ & 0.061 & 0.404\tabularnewline
$\text{CM}_{n}$ & 0.072 & 0.401\tabularnewline
$\text{AD}_{n}$ & 0.440 & 0.421\tabularnewline
$\text{WA}_{n}$ & 0.043 & 0.680\tabularnewline
$\text{G}_{n}$ & 0.325 & 0.147\tabularnewline
\hline
\end{tabular}\label{P-values}
\end{center}
\end{table}

\section{Conclusion and open questions\label{sec:Conclusion}}
We have shown a new characterisation of the logistic law and proposed a weighted affine invariant $L^2$-type test. Monte Carlo results show that it is competitive to the state-of-the-art procedures. Asymptotic properties have been derived, including the limit null distribution and consistency against a large class of alternatives. We conclude the paper by pointing out open questions for further research.

Following the methodology in \cite{BEH:2017}, we have under alternatives satisfying a weak moment condition
$\sqrt{n}\left(T_n/n -\Delta\right)\cd \mbox{N}(0,\tau^2)$, as $n\rightarrow\infty$ where $\tau^2>0$ is a specified variance, for details, see Theorem 1 in \cite{BEH:2017}. Since the calculations are too involved to get further insights, we leave the derivation of formulas open for further research. Note that such results can lead to confidence intervals for $\Delta$ or approximations of the power function, for examples of such results see \cite{DEH:2021} and \cite{EHS:2020} in the multivariate normality setting.

Due to the increasing popularity of the log-logistic distribution in survival analysis, another avenue for future research is to adapt our test for scenarios where censoring is present. One possibility is to estimate the expected value in (3) by estimating the law of the survival times by the well-known Kaplan-Meier estimate. Some work on this has been done in the case of testing for exponentiality (see, e.g., \citealp{bojana2021} and \citealp{bothma2021}).

\bibliographystyle{apalike2}
\bibliography{references}

\begin{appendix}
\section{Asymptotic representation of estimators}\label{app:linrep}
In this section we derive explicit formulae for the linear representations of the estimators in \eqref{eq:psi_11} and \eqref{eq:psi_21}, for comparison we refer to \citet{meintanis2004goodness}, p.313.
\subsection{Maximum-likelihood estimators}
The maximum-likelihood estimators $\widehat{\mu}^{ML}_n$ and $\widehat{\sigma}^{ML}_n$ of the parameters $\mu$ and $\sigma$ in \eqref{eq:density} satisfy the equations, see displays (23.35) and (23.36) in \cite{JKB:1995},
\begin{eqnarray*}
\sum_{j=1}^n\Big[1+\exp\left(\frac{X_j-\widehat{\mu}^{ML}_n}{\widehat{\sigma}^{ML}_n}\right)\Big]&=&\frac{n}2,\\
\sum_{j=1}^n\frac{X_j-\widehat{\mu}^{ML}_n}{\widehat{\sigma}^{ML}_n}\Bigg[\frac{1-\exp\left(\frac{X_j-\widehat{\mu}^{ML}_n}{\widehat{\sigma}^{ML}_n}\right)}{1+\exp\left(\frac{X_j-\widehat{\mu}^{ML}_n}{\widehat{\sigma}^{ML}_n}\right)}\Bigg]&=&n.
\end{eqnarray*}
An implementation is found in the \texttt{R}-package \texttt{EnvStats}, see \cite{M:2013}. Direct calculations show that the score vector of $X\sim \mbox{L}(\mu,\sigma)$, $\mu\in\R$, $\sigma>0$, is
\begin{equation*}
U_{(\mu,\sigma)}(X)=\frac1{\sigma^2}\tanh\left(\frac{X-\mu}{2\sigma}\right)\left(\sigma,(X-\mu)\right)^\top+(0,-\sigma^{-1})^\top,
\end{equation*}
where $x^\top$ stands for the transpose of a vector $x$. The Fisher information matrix is
\begin{equation*}
I_{(\mu,\sigma)}=\sigma^{-2}\left(\begin{array}{cc} \frac13 & 0 \\ 0 & \frac{\pi^2+3}{9} \end{array}\right),
\end{equation*}
which is easily inverted due to the diagonal form. By \cite{BD:2015}, Section 6.2.1, we hence have for $\mu=0$ and $\sigma=1$ the asymptotic expansions
\begin{equation*}
\sqrt{n}\widehat{\mu}^{ML}_n=\frac3{\sqrt{n}}\sum_{j=1}^n\tanh(X_j/2)+o_\PP(1)\;\mbox{and}\;\sqrt{n}(\widehat{\sigma}^{ML}_n-1)=\frac9{(\pi^2+3)\sqrt{n}}\sum_{j=1}^n\left(X_j\tanh(X_j/2)-1\right)+o_{\PP}(1).
\end{equation*}

\subsection{Moment estimators}
Since for $X\sim \mbox{L}(\mu,\sigma)$, $\mu\in\R$, $\sigma>0$, we have $\E[X]=\mu$ and $\mathbb{V}[X]=\pi^2\sigma^2/3$ the moment estimators $\widehat{\mu}^{ME}_n$ and $\widehat{\sigma}^{ME}_n$ are given by $\widehat{\mu}^{ME}_n=\frac1n\sum_{j=1}^nX_j=\overline{X}_n$ and $\widehat{\sigma}^{ME}_n=\frac{\sqrt{3}}{\pi}S_n$, where $S_n^2=\frac1n\sum_{j=1}^n (X_j-\overline{X}_n)^2$. An implementation is found in the \texttt{R}-package \texttt{EnvStats}, see \cite{M:2013}. By the same arguments as in \cite{BE:20}, p. 113, we have
\begin{equation*}
    \sqrt{n}\widehat{\mu}^{ME}_n=\frac1{\sqrt{n}}\sum_{j=1}^nX_j\quad\mbox{and}\quad \sqrt{n}(\widehat{\sigma}^{ME}_n-1)=\frac1{\sqrt{n}}\sum_{j=1}^n\frac12\left(\frac3{\pi^2} X_j^2-1\right)+o_\PP(1).
\end{equation*}
Note that the unbiased moment estimators use $\tilde{S}_n^2=\frac1{n-1}\sum_{j=1}^n (X_j-\overline{X}_n)^2$ instead of $S_n^2$.

\end{appendix}

\end{document}